\documentclass{amsart}
\usepackage{amsmath}
\usepackage{amssymb}
\usepackage{amsthm}

\usepackage{extarrows}
\usepackage{comment}
\usepackage{mathtools}
\usepackage{thmtools}
\usepackage{tikz}

\makeatletter
\@namedef{subjclassname@2020}{\textup{2020} Mathematics Subject Classification}
\makeatother

\usepackage{array}

\usepackage{arydshln}

\theoremstyle{plain}
\newtheorem{theorem}{Theorem}

\newtheorem*{theorem*}{Theorem}
\newtheorem{lemma}[theorem]{Lemma}

\newtheorem{proposition}[theorem]{Proposition}
\theoremstyle{definition}
\newtheorem{definition}[theorem]{Definition}

\newtheorem{conjecture}[theorem]{Conjecture}

\DeclareMathOperator{\Con}{Con}

\DeclareMathOperator{\FL}{FL}
\DeclareMathOperator{\OM}{OutMed}
\DeclareMathOperator{\IM}{InnMed}

\DeclareMathOperator{\SL}{SL}

\newcommand{\alg}[1]{\mathbf{#1}}
\newcommand{\rel}[1]{\mathbb{#1}}
\newcommand{\var}[1]{\mathcal{#1}}

\begin{document}

\title{Outer and inner medians in some small lattices}

\author{Leen Aburub}
\address{Bolyai Institute, University of Szeged, Szeged, Aradi V\'{e}rtan\'{u}k tere 1, HUNGARY 6720}
\email{aburub.Leen.Azeez@stud.u-szeged.hu}
%\address {Bolyai Institute, University of Szeged, Hungary}

\author{Gerg\H o Gyenizse}
\address{Department of Algebra and Number Theory, Bolyai Institute, University of Szeged, Szeged, Aradi V\'{e}rtan\'{u}k tere 1, HUNGARY 6720}
\email{gyenizse.gergo@math.u-szeged.hu}
%\address {Bolyai Institute, University of Szeged, Hungary}

\thanks{Supported by the \'UNKP-21-4-SZTE-512 New National Excellence Program and by the NKFIH grants K138892 and K153383 of the Ministry for Innovation and Technology from the source of the National Research, Development and Innovation Fund.}

\begin{abstract}
By median we mean a scheme that inputs three element of a lattice, and outputs an element that is an average of the three inputs in a certain sense. The medians of a given finite lattice form a new lattice that is usually larger than the original, but generates a (not necessarily strictly) smaller variety. A median is called inner if it is a term function. The inner median lattice is closely related to the symmetric part of the equational basis of the lattice. We determine the outer and inner median lattices of all lattices of at most six elements.
\end{abstract}

\keywords{lattices, free lattices, symmetrical terms, lattice median}

\subjclass[2020]{06B25, 06B20}

\maketitle

\section{Introduction}

When categorizing lattices, the usual starting point is historically modularity, then distributivity. More generally, the equational theory gives a great deal of information about the lattice, and the quasi-equational theory even more. By \cite{baker1972equational}, the equational theory of a finite lattice has a finite basis i.e. there is a finite set of identities held by the lattice so that all other held identities are a consequence of these.

In this paper, we are concerned by equational theories containing only symmetric terms. For example, distributivity can be famously characterized by the so-called {\em median law}: 
\[
(x\wedge y)\vee(x\wedge z)\vee(y\wedge z)\approx(x\vee y)\wedge(x\vee z)\wedge(y\vee z).
\]

Modularity, as will be seen later, also has such a characterization, though it is not as elegant.

In order to determine the symmetric equational theory of a lattice, we have to calculate the operations induced by symmetric term operations on it. We focus on the ternary case, as ternary lattice terms have a richness as large as those of higher arity (see \cite{freese1995free}), while being more manageable. We call such ternary operations (inner) medians. Note that this differs from the usual usage of this word, where median denotes any ternary operation satisfying the equational theory of the above function (see \cite{bandelt1983median}).

In this paper, we present some theoretical foundation of medians, and calculate the medians of small (at most $6$ elements) lattices. In doing so, we further our understanding of free lattices, structures that interest algebraists for almost a century--the word problem for lattices was solved by Whitman in 1941 \cite{whitman1941free}.

\section{Preliminaries and general observations}

For any $n\in\rel N$, we denote by $\FL(n)$ the lattice freely generated by the variables $x_1,\dots,x_n$, and by $\FL(\omega)$ the countably generated free lattice. The elements of $\FL(n)$ are called {\em $n$-ary lattice terms}. Each lattice term naturally induces an $n$-ary operation on any lattice. We usually do not distinguish between a term and its induced function if it does not lead to confusion. 

For any $t\in\FL(n)$ and permutation $\sigma\in S_n$ we denote by $t_{\sigma}$ the term obtained from $t$ by substituting all occurrences of $x_i$ with $x_{\sigma(i)}$. (In other words, $t_{\sigma}$ is the image of $t$ by the automorphism of $\FL(n)$ induced by the map $x_i\mapsto x_{\sigma(i)}$ on the free generator). We call a term $t$ {\em symmetric} if $t=t_{\sigma}$ holds for any permutation $\sigma$.

By the {\em meet (join) symmetrization} of a term $t$ we mean $\bigwedge_{\pi}t:=\bigwedge_{\sigma\in S_n}t_{\sigma}$ ($\bigvee_{\pi}t:=\bigvee_{\sigma\in S_n}t_{\sigma}$). The letter $\pi$ here does not stand for anything particular, it means ``for all permutations''. Clearly, a term is symmetric if and only if it coincides with its meet (join) symmetrization.

The symmetric terms of $\FL(n)$ form a sublattice denoted by $\SL(n)$. (We do not define $\SL(\omega)$, as there are no symmetric terms in $\FL(\omega)$.) By \cite{czedli2019symmetric}, $\SL(n)$ is not finitely generated for $n\geq 3$.

A lattice identity $t\approx s$ is symmetric if both $t$ and $s$ are symmetric. The symmetrizations of an identity $t\approx s$ are $\bigwedge_{\pi}t\approx\bigwedge_{\pi}s$ and $\bigvee_{\pi}t\approx\bigvee_{\pi}s$. Note that these are consequences of the original identity. It is possible that they are strictly weaker than the original, for example, the meet symmetrization of the distributive law $x_1\wedge(x_2\vee x_3)\approx(x_1\wedge x_2)\vee(x_1\wedge x_3)$ will be a trivial identity.

For a lattice variety $\mathcal{V}$, the {\em symmetric part} of $\mathcal{V}$ is the variety axiomatized by the symmetric identities true in $\mathcal{V}$. This is denoted by $\mathcal{V}^{sym}$. For a natural $k$, by $\mathcal{V}_k^{sym}$ we denote the variety axiomatized by the symmetric identities of arity at most $k$ true in $\mathcal{V}$.

\begin{conjecture}
Each lattice variety $\mathcal{V}$ equals its symmetric part. 
\end{conjecture}

There is a stronger variant of this conjecture.

\begin{conjecture}
For any lattice variety axiomatized by identities of arity at most $k$, $\mathcal{V}=\mathcal{V}_k^{sym}$ holds.
\end{conjecture}

In this paper, we are concerned about the varieties $\mathcal{V}_3^{sym}$ with $\mathcal{V}$ being a variety generated by a finite lattice. In other words, we study the $3$-ary symmetric identities satisfied by a finite lattice.

Let $\SL^*(n):=\SL(n)\backslash\{x_1\wedge\dots\wedge x_n,x_1\vee \dots\vee x_n\}$. In \cite{czedli2019symmetric} we proved that any element $t\in\SL^*(n)$ is a near unanimity term, that is, on any lattice it induces a function satisfying \[
t(x,x,\dots,x,y)=t(x,x,\dots,y,x)=\dots=t(y,x,\dots,x)=x
\] for all $x$ and $y$. This suggests the following definition.

\begin{definition}
Let $\alg L$ be a lattice. A mapping $f:\,L^3\to L$ that is a symmetric monotone majority operation is called a {\em median} of $\alg L$. A median is {\em inner} if it is also a term function, and {\em outer} otherwise. The medians of $\alg L$ form a sublattice of $\alg L^{L^3}$, this will be called the {\em outer median lattice} of $\alg L$, and denoted by $\OM\alg L$. In this, the inner medians form a sublattice, this {\em inner median lattice} will be denoted by $\IM\alg L$.
\end{definition}

We call the kernel of the natural homomorphism from $\SL^*(3)$ to $\IM\alg L$ (mapping a term to its induced function) the {\em $\alg L$-cut} of $\SL^*(3)$, or the cut induced by $\alg L$. Notice that if $\mathcal{V}$ is the variety generated by $\alg L$, then the $\alg L$-cut contains the same information as $\mathcal{V}_3^{sym}$. Consequently, lattices with non-isomorphic inner median lattices can be distinguished by a $3$-ary symmetric identity. In particular:

\begin{proposition}
\label{generated_variety_determines_im}
If two lattices generate the same variety, then their inner median lattices are isomorphic.
\end{proposition}

It may happen that lattices with isomorphic inner median lattices induce different cuts.

By \cite{czedli2019symmetric}, the smallest element of $\SL^*(3)$ is $m:=(x_1\wedge x_2)\vee(x_1\wedge x_3)\vee(x_2\wedge x_3)$ and the largest is its dual, $M:=(x_1\vee x_2)\wedge(x_1\vee x_3)\wedge(x_2\vee x_3)$. We call these (and the functions induced by these) the {\em lower and upper medians}, respectively. The lower median is also the smallest element in the outer median lattice of any $\alg L$, because if $t\in\OM\alg L$, then by the monotonicity and majority of $t$, 
\[
t(l_1,l_2,l_3)\geq t(l_1\wedge l_2,l_1\wedge l_2,l_3)=l_1\wedge l_2,
\]
likewise $t(l_1,l_2,l_3)\geq l_1\wedge l_3$ and $t(l_1,l_2,l_3)\geq l_2\wedge l_3$, hence $t(l_1,l_2,l_3)\geq m(l_1,l_2,l_3)$. Similarly, $M$ is the largest element of the outer median lattice of any $\alg L$.

\section{On the outer median lattice}

\begin{definition}
Let $\alg L$ be a lattice. For each triple $(l_1,l_2,l_3)\in L^3$, we call \newline $[m(l_1,l_2,l_3),M(l_1,l_2,l_3)]$ the {\em permitted interval} of this triple.

The poset $\mathcal{T}_{\alg L}$ (called the {\em T-poset} of $\alg L$) is defined the following way: its underlying set is the set of three-element subsets of $\alg L$ having nontrivial permitted interval. The order on $\mathcal{T}_{\alg L}$ is defined by 
\[
(l_1,l_2,l_3)\leq(k_1,k_2,k_3)\,\Leftrightarrow\,\exists\,\sigma\in S_{\{1,2,3\}}:\,\forall\, 1\leq i\leq 3:\,l_i\leq k_{\sigma(i)}.
\]

An order homomorphism $\mathcal{T}_{\alg L}\to\alg L$ is called {\em permitted} if it maps each triple to an element of its permitted interval.
\end{definition}

There is a natural bijection between the medians of $\alg L$ and the permitted homomorphisms. (By Theorem 2.51 of \cite{mckenzie1987algebras} $\mathcal{T}_{\alg L}$ is empty if and only if $\alg L$ is distributive. In this case, the empty set is such a homomorphism, and indeed, $\OM\alg L$ is trivial.)

When calculating $\mathcal{T}_{\alg L}\to\alg L$, it is helpful to remember that all triples in it contain either three pairwise incomparable elements, or two comparable elements and a third that is incomparable to the other two.

With the above description of medians, we can characterize the lattices with the next fewest medians.

\begin{theorem}
\label{th:two_outer_medians}
For a finite lattice $\alg L$, the following are equivalent.
\begin{enumerate}
    \item $|\OM\alg L|\leq 2$,
    \item $\OM\alg L=\IM\alg L$,
    \item of the $3$-generated sublattices of $\alg L$, at most one is nondistributive, and if there is a nondistributive $3$-generated sublattice, then it is isomorphic to $\alg N_5$. 
\end{enumerate}
\end{theorem}
\begin{proof}
{\color{white}a}
\begin{itemize}
    \item $(1)\Rightarrow(2)$: As the smallest and largest elements of $\OM\alg L$ are inner medians, this is obvious.
    
    \item $(2)\Rightarrow(3)$: Suppose first $\alg M_3$ is in the variety generated by $\alg L$. As $\alg L$ is finite, this implies by J\'onsson's Lemma that $\alg L$ has a sublattice $\alg L'$ that has a congruence $\mu$ so that $\alg L'\slash\mu\cong\alg M_3$. Choose $l_1,l_2,l_3\in L'$ so that
    \[
    [l_1\slash\mu,l_2\slash\mu,l_3\slash\mu]=\alg L'\slash\mu,
    \] 
    and define 
    \[
    l^*_1:=(m(l_1,l_2,l_3)\vee l_1)\wedge M(l_1,l_2,l_3).
    \] 
    Now $l^*_1$ is in the $\mu$-class of $l_1$, and it is in the permitted interval of $\{l_1,l_2,l_3\}$. We define the mapping 
    \[
    \mathcal{T}_{\alg L}\to\alg L:\,\{k_1,k_2,k_3\}\mapsto\begin{cases}
        M(k_1,k_2,k_3),\,\text{if $\{k_1,k_2,k_3\}>\{l_1,l_2,l_3\}$}\\
        l^*_1,\,\text{if $\{k_1,k_2,k_3\}=\{l_1,l_2,l_3\}$},\\
        m(k_1,k_2,k_3),\,\text{if $\{k_1,k_2,k_3\}\not\geq\{l_1,l_2,l_3\}$}
    \end{cases},
    \] 
    this is a permitted homomorphism, and it corresponds to an outer median, as an inner median must map $(l_1,l_2,l_3)$ to the class of either $m(l_1,l_2,l_3)$ or $M(l_1,l_2,l_3)$.
    
    Now suppose that $(l_1,l_2,l_3)$ and $(l_4,l_5,l_6)$ both generate sublattices isomorphic to $\alg N_5$. We can assume that $(l_1,l_2,l_3)\not\leq(l_4,l_5,l_6)$, and define the mapping 
    \[
    \mathcal{T}_{\alg L}\to\alg L:\,\{k_1,k_2,k_3\}\mapsto\begin{cases}
        M(k_1,k_2,k_3),\,\text{if $\{k_1,k_2,k_3\}\geq\{l_1,l_2,l_3\}$}\\
        m(k_1,k_2,k_3),\,\text{if $\{k_1,k_2,k_3\}\not\geq\{l_1,l_2,l_3\}$}
    \end{cases},
    \] 
    again, this is a permitted homomorphism. The corresponding median maps $\{l_1,l_2,l_3\}$ to $M(l_1,l_2,l_3)$ and $\{l_4,l_5,l_6\}$ to $m(l_4,l_5,l_6)$, and so must be an outer median.
    
    By Proposition \ref{generated_variety_determines_im}, the above argument also works if we only assume both $(l_1,l_2,l_3)$ and $(l_4,l_5,l_6)$ to generate a lattice that generates the same variety as $\alg N_5$.
    
    There is one outstanding case: when there are elements $l_1,l_2,l_3$ generating a sublattice $\alg S$ that is nondistributive, but generates a different variety than $\alg N_5$. That variety is contained in the one generated by $\alg L$, so it cannot contain $\alg M_3$. By \cite{jipsen2006varieties} (citing \cite{jonsson1979lattice}), it must cover one of fifteen subdirectly irreducible lattices ($\alg L_1,\dots,\alg L_{15}$). By J\'onsson's Lemma, this means that there is a $1\leq i\leq 15$ so that $\alg L_i$ is isomorphic to $\alg T\slash\nu$ for a sublattice $\alg T$ of $\alg S$ and a congruence $\nu$ of $\alg T$.
    
    It is easy to check that all the $\alg L_i$ contain at least two sublattices isomorphic to $\alg N_5$. Notice that if $t_1\slash\nu$, $t_2\slash\nu$ and $t_3\slash\nu$ generate a sublattice of $\alg T\slash\nu$ isomorphic to $\alg N_5$, then there are elements $t'_1,t'_2,t'_3\in T$ such that $t'_j\slash\nu=t_j\slash\nu$ for all $j=1,2,3$, and $t_1$, $t_2$ and $t_3$ generate a sublattice of $\alg T$ isomorphic to $\alg N_5$. (If $t_1\slash\nu<t_3\slash\nu$ in $\alg T\slash\nu$, then let $t'_3$ be the smallest element of the $\nu$-class of $t_3$, $t'_1$ the largest element of $t_1\slash\nu$ smaller than $t'_3$, and $t'_2=t_2$.) Thus, $\alg T$, and so $\alg L$, contains at least two sublattices isomorphic to $\alg N_5$. As we have seen, this means that there is an outer median of $\alg L$.
    
    \item $(3)\Rightarrow(1)$: If $\alg L$ is distributive, then $\OM\alg L$ is trivial. Otherwise, $\mathcal{T}_{\alg L}$ has only one element, denote it by $\{a,b,c\}$ with $a<c$. Notice that if $a<c'<c$, then $\{a,b,c'\}$ also generates a sublattice isomorphic to $\alg N_5$, contradicting the assumption. Therefore $a\prec c$, so the permitted interval of the only element of $\mathcal{T}_{\alg L}$ is a two-element interval. Hence, $|\OM\alg L|=2$.
\end{itemize}
\end{proof}

There is no constraint on the size of the outer median lattice, as the family of lattices $\alg E_n$ seen on Figure \ref{n medians} shows. Note that the T-poset of $\alg E_n$ is a chain of $n-1$ elements, each having $[a,b]$ as its permitted interval. Hence the permitted homomorphisms of $\alg E_n$ map elements of a (possibly empty) ideal of this chain to $a$, and elements of the complement filter to $b$. Thus, $\alg E_n$ has precisely $n$ medians, and $\OM\alg E_n$ is a chain.

Not all lattices will be isomorphic to an outer median lattice, an example is the six-element lattice $\alg M_4$ having four pairwise incomparable elements.

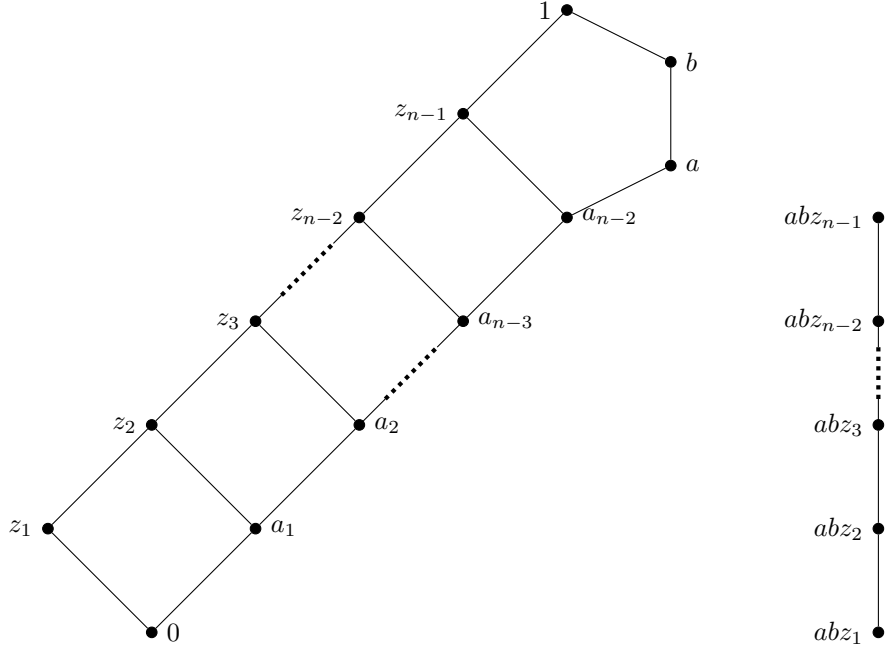
\begin{figure}
\tikzstyle{B} = [circle,draw=black,fill=black,minimum size=4pt,inner sep=0pt]
\tikzstyle{W} = [circle,draw=white,fill=white,minimum size=4pt,inner sep=0pt]
\tikzstyle{G} = [circle,draw=black,fill=green,minimum size=4pt,inner sep=0pt]
\tikzstyle{R} = [circle,draw=black,fill=red,minimum size=4pt,inner sep=0pt]

\begin{tikzpicture}[scale=0.69,auto=left]
  \draw (0,0) node[B] (0) [label=right:$0$] {};
  \draw (-2,2) node[B] (z1) [label=left:$z_1$] {};
  \draw (2,2) node[B] (a1) [label=right:$a_1$] {};
  \draw (0,4) node[B] (z2) [label=left:$z_2$] {};
  \draw (4,4) node[B] (a2) [label=right:$a_2$] {};
  \draw (2,6) node[B] (z3) [label=left:$z_3$] {};
  \draw (6,6) node[B] (an3) [label=right:$a_{n-3}$] {};
  \draw (4,8) node[B] (zn2) [label=left:$z_{n-2}$] {};
  \draw (6,10) node[B] (zn1) [label=left:$z_{n-1}$] {};
  \draw (8,8) node[B] (an2) [label=right:$a_{n-2}$] {};
  \draw (10,9) node[B] (a) [label=right:$a$] {};
  \draw (10,11) node[B] (b) [label=right:$b$] {};
  \draw (8,12) node[B] (1) [label=left:$1$] {};

 \draw (14,0) node[B] (00) [label=left:$abz_1$] {};
  \draw (14,2) node[B] (01) [label=left:$abz_2$] {};
  \draw (14,4) node[B] (02) [label=left:$abz_3$] {};
  \draw (14,6) node[B] (03) [label=left:$abz_{n-2}$] {};
  \draw (14,8) node[B] (04) [label=left:$abz_{n-1}$] {};

\foreach \from/\to in {0/z1, a1/z2, a2/z3, an3/zn2, an2/zn1, 0/a1, z1/z2, z2/z3, a1/a2, an3/an2, zn2/zn1, an2/a, a/b, b/1, zn1/1, 00/01, 01/02,03/04}
    \draw[-] (\from) -- (\to);
    \draw[-] (2,6) -- (2.5,6.5);
    \draw[-] (3.5,7.5) -- (4,8);
    \draw[-] (4,4) -- (4.5,4.5);
    \draw[-] (5.5,5.5) -- (6,6);
    \draw[-] (14,4) -- (14,4.5);
    \draw[-] (14,5.5) -- (14,6);
    \draw[ultra thick, dotted] (14,4.5) -- (14,5.5);
    \draw[ultra thick, dotted] (2.5,6.5) -- (3.5,7.5);
    \draw[ultra thick, dotted] (4.5,4.5) -- (5.5,5.5);
\end{tikzpicture}
\caption{The lattice $\alg E_{n}$ with $n$ medians and its T-poset}
\label{n medians}
\end{figure}

Clearly, $\OM\alg L$ is in the variety generated by $\alg L$. A little more can be said.

\begin{definition}
Let $\mathcal{V}_1,\mathcal{V}_2$ be lattice varieties. Then the {\em Mal'tsev product} $\mathcal{V}_1\circ\mathcal{V}_2$ is defined by 
\[
\alg L\in\mathcal{V}_1\circ\mathcal{V}_2\,\Leftrightarrow\,\exists\,\theta\in\Con\alg L:\,\alg L\slash\theta\in\mathcal{V}_2,\,\text{all the $\theta$-classes are in $\mathcal{V}_1$}
\]

In the case when $\mathcal{V}_1$ is the variety of distributive lattices, we use the notation $\mathcal{V}_2^d$ instead of $\mathcal{V}_1\circ\mathcal{V}_2$. 
\end{definition}

Note that we are using the restricted Mal'tsev product, i.e. we require elements of $\mathcal{V}_1\circ\mathcal{V}_2$ to be lattices, and not only algebras of the same similarity type otherwise satisfying the condition.

The Mal'tsev product of two lattice varieties is a quasivariety \cite{gratzer1985products}. It can be, but it rarely is, a variety, in particular $\var V^d$ cannot be a variety if $\var V$ is not modular \cite{harrison1985p}.

\begin{definition}
For any lattice $\alg L$, define the congruence $\theta^d\in\Con\alg L$ as the congruence generated by all the pairs $(m(a,b,c),M(a,b,c))$, where $a,b,c\in L$.
\end{definition}

Note that $\theta^d$ is the smallest congruence of the lattice such that the corresponding factor lattice is distributive.

\begin{proposition}
\label{p:smallest_dist_cong}
For any lattice variety $\mathcal{V}$ and lattice $\alg L$, 
\[
\alg L\in\mathcal{V}^d\,\Leftrightarrow\,\text{all the $\theta^d$-classes of $\alg L$ are in $\mathcal{V}$}.
\]
\end{proposition}

\begin{proof}
Suppose that $\alg L\in\mathcal{V}^d$, then there is a congruence $\theta$ of $\alg L$ such that $\alg L\slash\theta$ is distributive and the $\theta$-classes are in $\mathcal{V}$. As the lower and upper medians of $\alg L\slash\theta$ coincide: \[
m(a,b,c)\slash\theta=m(a\slash\theta,b\slash\theta,c\slash\theta)=M(a\slash\theta,b\slash\theta,c\slash\theta)=M(a,b,c)\slash\theta
\] for all $a,b,c\in L$, and hence $\theta^d\leq\theta$. Therefore, the $\theta^d$-classes are sublattices of $\theta$-classes, and thus are in $\mathcal{V}$.

The other direction is trivial.
\end{proof}

\begin{theorem}
Suppose that $\mathcal{V}$ is a lattice variety, and $\alg L\in\mathcal{V}^d$. Then $\OM\alg L\in\mathcal{V}$.
\end{theorem}

\begin{proof}
Suppose that $f\in\OM\alg L$, and $a,b,c\in L$. As $f$ is monotone and a majority operation, $f(a,b,c)\geq f(a\wedge b,a\wedge b,c)=a\wedge b$, $f(a,b,c)\geq f(a\wedge c,b,a\wedge c)=a\wedge c$, and $f(a,b,c)\geq f(a,b\wedge c,b\wedge c)=b\wedge c$. Thus $f(a,b,c)\geq m(a,b,c)$, and similarly, $f(a,b,c)\leq M(a,b,c)$. So, any median of $\alg L$ maps the triple $(a,b,c)$ to the interval $[m(a,b,c),M(a,b,c)]$. So \[
\OM\alg L\leq\Pi_{(a,b,c)\in L^3}[m(a,b,c),M(a,b,c)].
\] By Proposition \ref{p:smallest_dist_cong}, the lattice $[m(a,b,c),M(a,b,c)]$ is in $\mathcal{V}$, and therefore $\OM\alg L$ is also in $\mathcal{V}$.
\end{proof}

The converse is not true. Freese and McKenzie's example in \cite{freese2017maltsev} shows this: each permitted interval of the lattice $\alg M_{33}$ is either trivial or isomorphic to $\alg M_3$, so $\OM(\alg M_{33})\in V(\alg M_3)$. But $\alg M_{33}$ is simple, so it would only be in $V(\alg M_3)^d$ if it would be in $V(\alg M_3)$, which is not the case by J\'onsson's Lemma \cite{foundation}.

\section{Outer and inner median lattices of small lattices}

For a given finite lattice $\alg L$, we usually calculate $\OM\alg L$ by calculating its T-poset, then the permitted intervals and the permitted homomorphisms. However, sometimes we can simplify this process.

By the {\em linear sum} of the lattices $\alg L_1$ and $\alg L_2$, we mean the lattice with underlying set $(L_1\times\{1\})\cup(L_2\times\{2\})$ and order 
\[
(l,i)\leq (k,j)\,\Leftrightarrow\,i<j\,\vee(i=j\,\wedge\,l\leq k).
\] 
Usually we implicitly assume $L_1$ and $L_2$ to be disjoint and we omit the second coordinates.

By {\em gluing} $\alg L_2$ to $\alg L_1$, we mean the lattice obtained by factoring the linear sum of $\alg L_1$ and $\alg L_2$ with the congruence generated by $((1_{\alg L_1},1),(0_{\alg L_2},2))$, which clearly only contains singletons and a single two-element class.

\begin{proposition}
\label{glued_lattices}
If $\alg L$ is either the linear sum of the lattices $\alg L_1$ and $\alg L_2$, or is obtained by gluing $\alg L_2$ to $\alg L_1$, then 
\[
\OM\alg L\cong\OM\alg L_1\times\OM\alg L_2
\]
\end{proposition}

\begin{proof}

If $(k_1,k_2,k_3)\in(L_1\cup L_2)^3\backslash(L_1^3\cup L_2^3)$, then there is either a smallest or a largest element in the set $\{k_1,k_2,k_3\}$, which means $\{k_1,k_2,k_3\}\not\in\mathcal{T}_{\alg L}$. Therefore, 
\[
\mathcal{T}_{\alg L}=\mathcal{T}_{\alg L_1}\cup\mathcal{T}_{\alg L_2}.
\] 
Any element of the permitted interval of some triple in $\mathcal{T}_{\alg L_1}$ is smaller or equal than any element of the permitted interval of some triple in $\mathcal{T}_{\alg L_2}$. Thus, the set-theoretic join of a permitted homomorphism of $\alg L_1$ and a permitted homomorphism of $\alg L_2$ is a permitted homomorphism of $\alg L$, and we are ready. 
\end{proof}

The naive way of calculating $\IM\alg L$ is harder than calculating $\OM\alg L$. It involves determining the set of ternary terms of $\alg L$, that is, computing the subalgebra of $\alg L^{L^3}$ generated by the projections. We have two general ways of reducing this calculation.

\begin{lemma}
    Suppose that $\alg L$ is a subdirect product $\alg L\leq_{sd}\alg L_1\times\alg L_2$, and denote the $\alg L_i$-cut by $\theta_i$ ($i=1,2$). Then the $\alg L$-cut is $\theta_1\wedge\theta_2$.
\end{lemma}

\begin{proof}
    Take terms $s_1,s_2\in\SL^*(3)$. Then $s_1$ and $s_2$ are in the same class of the $\theta$-cut iff $s_1\approx s_2$ holds in $\alg L$. This happens precisely if $s_1\approx s_2$ holds in both $\alg L_1$ and $\alg L_2$, as $\alg L_1$ and $\alg L_2$ are in the variety generated by $\alg L$ and $\alg L$ is in the variety generated by $\alg L_1$ and $\alg L_2$. Hence, $(s_1,s_2)\in\theta$ iff both $(s_1,s_2)\in\theta_1$ and $(s_1,s_2)\in\theta_2$.
\end{proof}

Note that this includes the cases when $\alg L$ is the linear sum of $\alg L_1$ and $\alg L_2$, or is obtained by their gluing.

\begin{lemma}
    Suppose that $C$ is the set of $3$-generated sublattices of $\alg L$. For each $\alg L_i\in C$, let $\theta_i$ be the $\alg L_i$-cut. Then the $\alg L$-cut is $\bigwedge_{\alg L_i\in C}\theta_i$.
\end{lemma}
\begin{proof}
    Take terms $s_1,s_2\in\SL^*(3)$. If $(s_1,s_2)\in\theta$, then $s_1\approx s_2$ holds in $\alg L$, therefore it also holds on $\alg L_i$ for all $\alg L_i\in C$. So $(s_1,s_2)\in\theta_i$.

    Now assume that $(s_1,s_2)\in\theta_i$ for all $i$. Take arbitrary elements $l_1,l_2,l_3\in L$. Then the sublattice $\alg K=[l_1,l_2,l_3]$ of $\alg L$ is in $C$, thus $(s_1,s_2)$ is in the cut corresponding to $\alg K$. Hence, $s_1\approx s_2$ holds in $\alg K$, so $s_1(l_1,l_2,l_3)=s_2(l_1,l_2,l_3)$. Therefore, $s_1\approx s_2$ also holds in $\alg L$, thus $(s_1,s_2)\in\theta$. 
\end{proof}

Now in order to calculate the inner median lattices for all lattices of a fixed size, it is enough to determine the cuts corresponding to the $3$-generated subdirectly irreducible lattices of at most that size. We will do this for lattices of size at most $6$. Then we will only have to deal with the following $3$-generated subdirectly irreducible lattices:
\begin{enumerate}
    \item $\alg M_3$: this has five medians, which are determined by the image of the triple $abc$. If this image is $a$, $b$, or $c$, then the median is outer, as a term function needs to commute with any automorphism of the lattice. The other two medians are inner, as they are induced by the lower and upper medians, respectively. Note that by \cite{czedli2019symmetric}, the upper block of the $\alg M_3$-cut does not have a largest element, nor does the lower block have a largest.
    \item $\alg N_5$: this has a two-element inner median lattice by Theorem \ref{th:two_outer_medians}. By \cite{gyenizse2023symmetric}, the lowest element of the upper block of the $\alg N_5$-cut is $\bigvee_{\pi}(x\wedge(z\vee(x\wedge y)))$, and the largest element of the lower block is $t(x,y,z):=\bigwedge_{\pi}(x\vee(y\wedge(x\vee z)))$.
    \item $\alg L_4$: $\mathcal{T}_{\alg L_4}$ is a three-element poset: it has smallest element $abc$ with permitted interval $\{0,a,b,d\}$, and two incomparable elements, $acd$ and $bcd$ with permitted intervals $\{a,d\}$ and $\{b,d\}$, respectively. The outer median lattice of $\alg L_4$ turns out to be isomorphic to $\mathbf{3}^2$.
    
    The smallest and largest elements of $\alg L_4$, $0ab$ and $ddd$ are necessarily inner medians. Consider the term \[
    t(x,y,z):=\bigwedge_{\pi}(x\vee(y\wedge(x\vee z))),
    \] which we obtained from the $\alg N_5$-cut. As $\overline{t}(a,b,c)=0$ and $\overline{t}(a,c,d)=\overline{t}(b,c,d)=d$ hold, $\overline{t}$ ensures that $0dd$ is an inner median.
    
    In $\alg L_4$, both $\{a,c,d\}$ and $\{b,c,d\}$ generate sublattices isomorphic to $\alg N_5$. An inner median must act the same on both of these, so if must map either both or neither to $d$. Hence, $0db$, $0ad$, $bdb$ and $aad$ are outer medians.
    
    Finally, as $\alg L_4$ has an automorphism swapping $a$ and $b$, and $\{a,b,c\}$ is invariant under this, any inner median must map $\{a,b,c\}$ to an element fixed by this automorphism. Therefore, $bdd$ and $add$ are outer medians, and $\IM\alg L_4\cong\mathbf{3}$.
    \item $\alg L_5$: As it is the dual of $\alg L_4$, $\OM\alg L_5\cong\mathbf{3}^2$ and $\IM\alg L_5\cong\mathbf{3}$.
\end{enumerate}

The outer and inner median lattices of $6$ elements lattice are described in Table ~\ref{6elementsTable}.
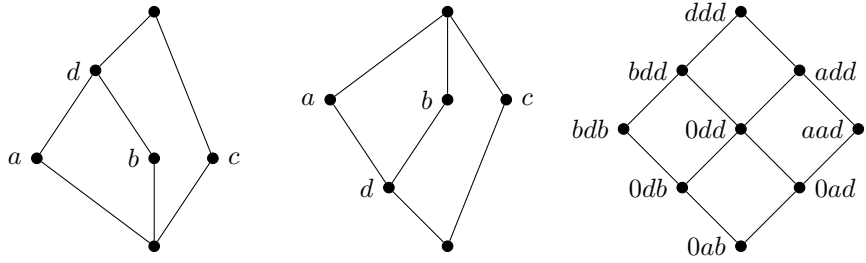
\begin{figure}[h]
\tikzstyle{B} = [circle,draw=black,fill=black,minimum size=4pt,inner sep=0pt]
\tikzstyle{W} = [circle,draw=white,fill=white,minimum size=4pt,inner sep=0pt]
\tikzstyle{G} = [circle,draw=black,fill=green,minimum size=4pt,inner sep=0pt]
\tikzstyle{R} = [circle,draw=black,fill=red,minimum size=4pt,inner sep=0pt]

\begin{tikzpicture}[scale=0.39,auto=left]
  \draw (14,8) node[B] (0) [label=right:] {};
  \draw (10,5) node[B] (a) [label=left:$a$] {};
  \draw (14,5) node[B] (b) [label=left:$b$] {};
  \draw (16,5) node[B] (c) [label=right:$c$] {};
  \draw (12,2) node[B] (d) [label=left:$d$] {};
  \draw (14,0) node[B] (1) [label=left:] {};
  \draw (4,0) node[B] (0d) [label=right:] {};
  \draw (0,3) node[B] (ad) [label=left:$a$] {};
  \draw (4,3) node[B] (bd) [label=left:$b$] {};
  \draw (6,3) node[B] (cd) [label=right:$c$] {};
  \draw (2,6) node[B] (dd) [label=left:$d$] {};
  \draw (4,8) node[B] (1d) [label=left:] {};
  \draw (24,0) node[B] (0ab) [label=left:$0ab$] {};
  \draw (22,2) node[B] (0db) [label=left:$0db$] {};
  \draw (26,2) node[B] (0ad) [label=right:$0ad$] {};
  \draw (20,4) node[B] (bdb) [label=left:$bdb$] {};
  \draw (24,4) node[B] (0dd) [label=left:$0dd$] {};
  \draw (28,4) node[B] (aad) [label=left:$aad$] {};
  \draw (22,6) node[B] (bdd) [label=left:$bdd$] {};
  \draw (26,6) node[B] (add) [label=right:$add$] {};
  \draw (24,8) node[B] (ddd) [label=left:$ddd$] {};

\foreach \from/\to in {0/a,0/b,0/c,a/d,b/d,d/1,c/1,0d/ad,0d/bd,0d/cd,ad/dd,bd/dd,dd/1d,cd/1d,0ab/0db,0ab/0ad,0db/bdb,0db/0dd,0ad/0dd,0ad/aad,bdb/bdd,0dd/bdd,0dd/add,aad/add,bdd/ddd,add/ddd}
    \draw[-] (\from) -- (\to);
\end{tikzpicture}
\caption{The lattices $\alg L_4$ and $\alg L_5$, and the poset $\mathcal{T}_{\alg L_4}$}
\end{figure}

\section{Symmetric characterization of modularity}

The $\alg N_5$-cut naturally yields a symmetric identity characterizing modularity:

\begin{theorem}
A lattice is modular if and only if it satisfies the inequality \[
\bigvee_{\pi}(x\wedge(z\vee(x\wedge y)))\leq\bigwedge_{\pi}(x\vee(z\wedge(x\vee y)))
\]
\end{theorem}
\begin{proof}
    Note that the term on the left hand side is the smallest element of the upper block of the $\alg N_5$-cut, and the term on the right hand side is the largest element of the lower block. This means that the inequality does not hold in $\alg N_5$, hence it does hold in any nonmodular lattice.

    On the other hand, in a modular lattice, 
    \[
    x\wedge(z\vee(x\wedge y))\approx (x\wedge z)\vee(x\wedge y)
    \] 
    holds by definition. By symmetrization, then 
    \[
    \bigvee_{\pi}(x\wedge(z\vee(x\wedge y)))\approx\bigvee_{\pi}((x\wedge z)\vee(x\wedge y))\approx \bigvee_{\pi}(x\wedge y),
    \] 
    the second identity holding in any lattice. Similarly, 
    \[
    \bigwedge_{\pi}(x\vee(z\wedge(x\vee y)))\approx\bigwedge_{\pi}(x\vee y),
    \] 
    and as 
    \[
    \bigvee_{\pi}(x\wedge y)\leq\bigwedge_{\pi}(x\vee y)
    \] 
    again holds in any lattice, we are done.
\end{proof}

\begin{table}[]

\begin{tabular}{|m{0.9cm}|m{2.1cm}|m{2.8cm}|m{2.5cm}|m{1.4cm}|}

 \hline
Name & Lattice &    ${\color{white}aaaaaa}\mathcal{T}_L$   &$\OM L$ & $\IM L$\\
\hline

$\alg M_3$ & 
\tikzstyle{B} = [circle,draw=black,fill=black,minimum size=2pt,inner sep=0pt]
\tikzstyle{W} = [circle,draw=white,fill=white,minimum size=4pt,inner sep=0pt]
\tikzstyle{G} = [circle,draw=black,fill=green,minimum size=4pt,inner sep=0pt]
\tikzstyle{R} = [circle,draw=black,fill=red,minimum size=4pt,inner sep=0pt]

\begin{tikzpicture}[scale=0.14,auto=left]
  \draw (3,0) node[B] (0) [label=right:] {};
  \draw (0,3) node[B] (a) [label=left:{\scriptsize $a$}] {};
  \draw (3,3) node[B] (b) [label=left:{\scriptsize $b$}] {};
  \draw (6,3) node[B] (c) [label=right:{\scriptsize $c$}] {};
  \draw (3,6) node[B] (d) [label=left:{\scriptsize }] {};

\foreach \from/\to in {0/a,0/b,0/c,a/d,b/d,c/d}
    \draw[-] (\from) -- (\to);
\end{tikzpicture} &   
\tikzstyle{B} = [circle,draw=black,fill=black,minimum size=2pt,inner sep=0pt]
\tikzstyle{W} = [circle,draw=white,fill=white,minimum size=4pt,inner sep=0pt]
\tikzstyle{G} = [circle,draw=black,fill=green,minimum size=4pt,inner sep=0pt]
\tikzstyle{R} = [circle,draw=black,fill=red,minimum size=4pt,inner sep=0pt]

\begin{tikzpicture}[scale=0.14,auto=left]
  \draw (0,3) node[B] (a) [label=left:{\scriptsize $abc$}] {};
    \draw (-9,7) node[W] [] {};

\end{tikzpicture} & $\alg M_3$ & $\mathbf{2}$ \\
 \hline

 $\alg N_5$ & \tikzstyle{B} = [circle,draw=black,fill=black,minimum size=2pt,inner sep=0pt]
\tikzstyle{W} = [circle,draw=white,fill=white,minimum size=4pt,inner sep=0pt]
\tikzstyle{G} = [circle,draw=black,fill=green,minimum size=4pt,inner sep=0pt]
\tikzstyle{R} = [circle,draw=black,fill=red,minimum size=4pt,inner sep=0pt]

\begin{tikzpicture}[scale=0.14,auto=left]
  \draw (2,0) node[B] (0) [label=right:] {};
  \draw (0,2) node[B] (a) [label=left:{\scriptsize $a$}] {};
  \draw (4,3) node[B] (b) [label=right:{\scriptsize $b$}] {};
  \draw (0,4) node[B] (c) [label=left:{\scriptsize $c$}] {};
  \draw (2,6) node[B] (d) [label=left:{\scriptsize }] {};

\foreach \from/\to in {0/a,0/b,a/c,b/d,c/d}
    \draw[-] (\from) -- (\to);
\end{tikzpicture} &  
\tikzstyle{B} = [circle,draw=black,fill=black,minimum size=2pt,inner sep=0pt]
\tikzstyle{W} = [circle,draw=white,fill=white,minimum size=4pt,inner sep=0pt]
\tikzstyle{G} = [circle,draw=black,fill=green,minimum size=4pt,inner sep=0pt]
\tikzstyle{R} = [circle,draw=black,fill=red,minimum size=4pt,inner sep=0pt]

\begin{tikzpicture}[scale=0.14,auto=left]
  \draw (0,3) node[B] (a) [label=left:{\scriptsize $abc$}] {};
    \draw (-9,7) node[W] [] {};

\end{tikzpicture}
& $\alg 2$ & $\mathbf{2}$ \\
 \hline

$\alg M_4$ & 
\tikzstyle{B} = [circle,draw=black,fill=black,minimum size=2pt,inner sep=0pt]
\tikzstyle{W} = [circle,draw=white,fill=white,minimum size=4pt,inner sep=0pt]
\tikzstyle{G} = [circle,draw=black,fill=green,minimum size=4pt,inner sep=0pt]
\tikzstyle{R} = [circle,draw=black,fill=red,minimum size=4pt,inner sep=0pt]

\begin{tikzpicture}[scale=0.14,auto=left]
  \draw (4,0) node[B] (0) [label=right:] {};
  \draw (-1,4) node[B] (a) [label=left:{\scriptsize $a$}] {};
  \draw (2.5,4) node[B] (b) [label=left:{\scriptsize $b$}] {};
  \draw (5.5,4) node[B] (c) [label=right:{\scriptsize $c$}] {};
  \draw (9,4) node[B] (d) [label=right:{\scriptsize $d$}] {};
  \draw (4,8) node[B] (e) [label=left:] {};

\foreach \from/\to in {0/a,0/b,0/c,0/d,a/e,b/e,c/e,d/e}
    \draw[-] (\from) -- (\to);
\end{tikzpicture}
   &    
   \tikzstyle{B} = [circle,draw=black,fill=black,minimum size=2pt,inner sep=0pt]
\tikzstyle{W} = [circle,draw=white,fill=white,minimum size=4pt,inner sep=0pt]
\tikzstyle{G} = [circle,draw=black,fill=green,minimum size=4pt,inner sep=0pt]
\tikzstyle{R} = [circle,draw=black,fill=red,minimum size=4pt,inner sep=0pt]

\begin{tikzpicture}[scale=0.14,auto=left]
  \draw (0,3) node[B] (a) [label=left:{\scriptsize $abc$}] {};
  \draw (5,3) node[B] (a) [label=left:{\scriptsize $abd$}] {};
  \draw (10,3) node[B] (a) [label=left:{\scriptsize $acd$}] {};
  \draw (15,3) node[B] (a) [label=left:{\scriptsize $bcd$}] {};

\end{tikzpicture}
   &$\alg M_4^4$ & $\mathbf{2}$ \\
   \hline

      $\alg A_1$   & \tikzstyle{B} = [circle,draw=black,fill=black,minimum size=2pt,inner sep=0pt]
\tikzstyle{W} = [circle,draw=white,fill=white,minimum size=4pt,inner sep=0pt]
\tikzstyle{G} = [circle,draw=black,fill=green,minimum size=4pt,inner sep=0pt]
\tikzstyle{R} = [circle,draw=black,fill=red,minimum size=4pt,inner sep=0pt]

\begin{tikzpicture}[scale=0.14,auto=left]
  \draw (3.5,0) node[B] (0) [label=right:] {};
  \draw (1,2) node[B] (a) [label=left:{\scriptsize $a$}] {};
  \draw (1,4) node[B] (c) [label=left:{\scriptsize $c$}] {};
  \draw (3.5,3) node[B] (b) [label=right:{\scriptsize $b$}] {};
  \draw (7,3) node[B] (d) [label=right:{\scriptsize $d$}] {};
  \draw (3.5,6) node[B] (1) [label=left:] {};

\foreach \from/\to in {0/a,a/c,c/1,0/b,0/d,b/1,d/1}
    \draw[-] (\from) -- (\to);
\end{tikzpicture} &   
\tikzstyle{B} = [circle,draw=black,fill=black,minimum size=2pt,inner sep=0pt]
\tikzstyle{W} = [circle,draw=white,fill=white,minimum size=4pt,inner sep=0pt]
\tikzstyle{G} = [circle,draw=black,fill=green,minimum size=4pt,inner sep=0pt]
\tikzstyle{R} = [circle,draw=black,fill=red,minimum size=4pt,inner sep=0pt]

\begin{tikzpicture}[scale=0.14,auto=left]
  \draw (0,3) node[B] (a) [label=left:{\scriptsize $abc$}] {};
  \draw (6,3) node[B] (b) [label=left:{\scriptsize $acd$}] {};
  \draw (12,3) node[B] (c) [label=left:{\scriptsize $abd$}] {};
  \draw (12,7) node[B] (d) [label=left:{\scriptsize $bcd$}] {};

\foreach \from/\to in {c/d}
    \draw[-] (\from) -- (\to);
\end{tikzpicture}
&$\alg 2^2\times \alg A_1^\mathbf{2}$ & $\mathbf{2}^2$ \\
\hline

$\alg A_2$ & \tikzstyle{B} = [circle,draw=black,fill=black,minimum size=2pt,inner sep=0pt]
\tikzstyle{W} = [circle,draw=white,fill=white,minimum size=4pt,inner sep=0pt]
\tikzstyle{G} = [circle,draw=black,fill=green,minimum size=4pt,inner sep=0pt]
\tikzstyle{R} = [circle,draw=black,fill=red,minimum size=4pt,inner sep=0pt]

\begin{tikzpicture}[scale=0.14,auto=left]
  \draw (2,0) node[B] (0) [label=right:] {};
  \draw (0,2) node[B] (a) [label=left:{\scriptsize $a$}] {};
  \draw (4,2) node[B] (b) [label=right:{\scriptsize $b$}] {};
  \draw (0,4) node[B] (c) [label=left:{\scriptsize $c$}] {};
  \draw (4,4) node[B] (d) [label=right:{\scriptsize $d$}] {};
  \draw (2,6) node[B] (1) [label=left:] {};
  \draw (2,7) node[W] (11) [label=left:] {};

\foreach \from/\to in {0/a,a/c,c/1,0/b,b/d,d/1}
    \draw[-] (\from) -- (\to);
\end{tikzpicture}

&  

\tikzstyle{B} = [circle,draw=black,fill=black,minimum size=2pt,inner sep=0pt]
\tikzstyle{W} = [circle,draw=white,fill=white,minimum size=4pt,inner sep=0pt]
\tikzstyle{G} = [circle,draw=black,fill=green,minimum size=4pt,inner sep=0pt]
\tikzstyle{R} = [circle,draw=black,fill=red,minimum size=4pt,inner sep=0pt]

\begin{tikzpicture}[scale=0.14,auto=left]
  \draw (0,5) node[W] (a) [] {};
  \draw (0,0) node[B] (a) [label=left:{\scriptsize $abc$}] {};
  \draw (0,4) node[B] (b) [label=left:{\scriptsize $acd$}] {};
  \draw (12,0) node[B] (c) [label=left:{\scriptsize $abd$}] {};
  \draw (12,4) node[B] (d) [label=left:{\scriptsize $bcd$}] {};
    \draw (0,6) node[W] [] {};

\foreach \from/\to in {c/d,a/b}
    \draw[-] (\from) -- (\to);
\end{tikzpicture}

&$\mathbf{3}^2$ & $\mathbf{2}$\\
\hline
$\alg A_3$ &\tikzstyle{B} = [circle,draw=black,fill=black,minimum size=2pt,inner sep=0pt]
\tikzstyle{W} = [circle,draw=white,fill=white,minimum size=4pt,inner sep=0pt]
\tikzstyle{G} = [circle,draw=black,fill=green,minimum size=4pt,inner sep=0pt]
\tikzstyle{R} = [circle,draw=black,fill=red,minimum size=4pt,inner sep=0pt]

\begin{tikzpicture}[scale=0.14,auto=left]
  \draw (3,0) node[B] (0) [label=right:] {};
  \draw (1,2) node[B] (a) [label=left:{\scriptsize $a$}] {};
  \draw (5.5,4) node[B] (b) [label=right:{\scriptsize $b$}] {};
  \draw (1,4) node[B] (c) [label=left:{\scriptsize $c$}] {};
  \draw (1,6) node[B] (d) [label=left:{\scriptsize $d$}] {};
  \draw (3,8) node[B] (1) [label=left:] {};

\foreach \from/\to in {0/a,a/c,c/d,d/1,0/b,b/1}
    \draw[-] (\from) -- (\to);
\end{tikzpicture} 
&   

\tikzstyle{B} = [circle,draw=black,fill=black,minimum size=2pt,inner sep=0pt]
\tikzstyle{W} = [circle,draw=white,fill=white,minimum size=4pt,inner sep=0pt]
\tikzstyle{G} = [circle,draw=black,fill=green,minimum size=4pt,inner sep=0pt]
\tikzstyle{R} = [circle,draw=black,fill=red,minimum size=4pt,inner sep=0pt]

\begin{tikzpicture}[scale=0.14,auto=right]

  \draw (9,0) node[B] (a) [label=left:{\scriptsize $abc$}] {};
  \draw (9,3) node[B] (c) [label=left:{\scriptsize $abd$}] {};
  \draw (9,6) node[B] (d) [label=left:{\scriptsize $bcd$}] {};
  \draw (0,8) node[white] (x) [label=left:] {};

\foreach \from/\to in {a/c,c/d}
    \draw[-] (\from) -- (\to);
\end{tikzpicture}

 &  \tikzstyle{B} = [circle,draw=black,fill=black,minimum size=2pt,inner sep=0pt]
\tikzstyle{W} = [circle,draw=white,fill=white,minimum size=4pt,inner sep=0pt]
\tikzstyle{G} = [circle,draw=black,fill=green,minimum size=4pt,inner sep=0pt]
\tikzstyle{R} = [circle,draw=black,fill=red,minimum size=4pt,inner sep=0pt]

\begin{tikzpicture}[scale=0.14,auto=left]

  \draw (14,0) node[B] (aab) [label=right:] {};
  \draw (12,2) node[B] (aac) [label=left:] {};
  \draw (16,2) node[B] (abb) [label=right:] {};
  \draw (14,4) node[B] (abc) [label=left:] {};
  \draw (12,6) node[B] (acc) [label=left:] {};
  \draw (18,4) node[B] (bbb) [label=right:] {};
  \draw (16,6) node[B] (bbc) [label=right:] {};
  \draw (14,8) node[B] (bcc) [label=right:] {};

\foreach \from/\to in {aab/aac,aab/abb,aac/abc,abb/abc,abb/bbb,abc/acc,abc/bbc,bbb/bbc,acc/bcc,bbc/bcc}
    \draw[-] (\from) -- (\to);
\end{tikzpicture}   & $\mathbf{2}$\\
\hline
 $\alg L_4$ & \tikzstyle{B} = [circle,draw=black,fill=black,minimum size=2pt,inner sep=0pt]
\tikzstyle{W} = [circle,draw=white,fill=white,minimum size=4pt,inner sep=0pt]
\tikzstyle{G} = [circle,draw=black,fill=green,minimum size=4pt,inner sep=0pt]
\tikzstyle{R} = [circle,draw=black,fill=red,minimum size=4pt,inner sep=0pt]

\begin{tikzpicture}[scale=0.15,auto=left]
  \draw (14,6) node[B] (0) [label=right:] {};
  \draw (11,4) node[B] (a) [label=left:{\scriptsize $a$}] {};
  \draw (14,4) node[B] (b) [label=left:{\scriptsize $b$}] {};
  \draw (16,4) node[B] (c) [label=right:{\scriptsize $c$}] {};
  \draw (12.5,2) node[B] (d) [label=left:{\scriptsize $d$}] {};
  \draw (14,0) node[B] (1) [label=left:] {};

\foreach \from/\to in {0/a,0/b,0/c,a/d,b/d,d/1,c/1}
    \draw[-] (\from) -- (\to);
\end{tikzpicture} &

\tikzstyle{B} = [circle,draw=black,fill=black,minimum size=2pt,inner sep=0pt]
\tikzstyle{W} = [circle,draw=white,fill=white,minimum size=4pt,inner sep=0pt]
\tikzstyle{G} = [circle,draw=black,fill=green,minimum size=4pt,inner sep=0pt]
\tikzstyle{R} = [circle,draw=black,fill=red,minimum size=4pt,inner sep=0pt]

\begin{tikzpicture}[scale=0.14,auto=left]

  \draw (0,0) node[B] (a) [label=left:{\scriptsize $acd$}] {};
  \draw (3,3) node[B] (c) [label=left:{\scriptsize $abc$}] {};
  \draw (6,0) node[B] (d) [label=right:{\scriptsize $bcd$}] {};

\foreach \from/\to in {a/c,c/d}
    \draw[-] (\from) -- (\to);
\end{tikzpicture}

& 

\tikzstyle{B} = [circle,draw=black,fill=black,minimum size=2pt,inner sep=0pt]
\tikzstyle{W} = [circle,draw=white,fill=white,minimum size=4pt,inner sep=0pt]
\tikzstyle{G} = [circle,draw=black,fill=green,minimum size=4pt,inner sep=0pt]
\tikzstyle{R} = [circle,draw=black,fill=red,minimum size=4pt,inner sep=0pt]

\begin{tikzpicture}[scale=0.16,auto=left]

  \draw (14,0) node[B] (aab) [label=left:{\scriptsize $0ab$}] {};
  \draw (12,2) node[B] (aac) [label=left:{\scriptsize $0bd$}] {};
  \draw (16,2) node[B] (abb) [label=right:{\scriptsize $0ad$}] {};
  \draw (14,4) node[B] (abc) [label=left:{\scriptsize $0dd$}] {};
  \draw (12,6) node[B] (acc) [label=left:{\scriptsize $bdd$}] {};
  \draw (18,4) node[B] (bbb) [label=right:{\scriptsize $aad$}] {};
  \draw (10,4) node[B] (a) [label=left:{\scriptsize $bdb$}] {};
  \draw (16,6) node[B] (bbc) [label=right:{\scriptsize $add$}] {};
  \draw (14,8) node[B] (bcc) [label=left:{\scriptsize $ddd$}] {};

\foreach \from/\to in {aab/aac,aab/abb,aac/abc,abb/abc,abb/bbb,abc/acc,abc/bbc,bbb/bbc,acc/bcc,bbc/bcc,aac/a,a/acc}
    \draw[-] (\from) -- (\to);
\end{tikzpicture}  

& $\mathbf{3}$ \\

 \hline

$\alg  L_5$ & \tikzstyle{B} = [circle,draw=black,fill=black,minimum size=2pt,inner sep=0pt]
\tikzstyle{W} = [circle,draw=white,fill=white,minimum size=4pt,inner sep=0pt]
\tikzstyle{G} = [circle,draw=black,fill=green,minimum size=4pt,inner sep=0pt]
\tikzstyle{R} = [circle,draw=black,fill=red,minimum size=4pt,inner sep=0pt]

\begin{tikzpicture}[scale=0.15,auto=left]

  \draw (4,0) node[B] (0d) [label=right:] {};
  \draw (1,2) node[B] (ad) [label=left:{\scriptsize $a$}] {};
  \draw (4,2) node[B] (bd) [label=left:{\scriptsize $b$}] {};
  \draw (6,2) node[B] (cd) [label=right:{\scriptsize $c$}] {};
  \draw (2.5,4) node[B] (dd) [label=left:{\scriptsize $d$}] {};
  \draw (4,6) node[B] (1d) [label=left:] {};

\foreach \from/\to in {0d/ad,0d/bd,0d/cd,ad/dd,bd/dd,dd/1d,cd/1d}
    \draw[-] (\from) -- (\to);
\end{tikzpicture} &   

\tikzstyle{B} = [circle,draw=black,fill=black,minimum size=2pt,inner sep=0pt]
\tikzstyle{W} = [circle,draw=white,fill=white,minimum size=4pt,inner sep=0pt]
\tikzstyle{G} = [circle,draw=black,fill=green,minimum size=4pt,inner sep=0pt]
\tikzstyle{R} = [circle,draw=black,fill=red,minimum size=4pt,inner sep=0pt]

\begin{tikzpicture}[scale=0.14,auto=left]

  \draw (0,0) node[B] (a) [label=left:{\scriptsize $acd$}] {};
  \draw (3,3) node[B] (c) [label=left:{\scriptsize $abc$}] {};
  \draw (6,0) node[B] (d) [label=right:{\scriptsize $bcd$}] {};

\foreach \from/\to in {a/c,c/d}
    \draw[-] (\from) -- (\to);
\end{tikzpicture}

& $\mathbf{3}^2$ & $\mathbf{3}$ \\

 \hline

 $\alg B_1$ & \tikzstyle{B} = [circle,draw=black,fill=black,minimum size=2pt,inner sep=0pt]
\tikzstyle{W} = [circle,draw=white,fill=white,minimum size=4pt,inner sep=0pt]
\tikzstyle{G} = [circle,draw=black,fill=green,minimum size=4pt,inner sep=0pt]
\tikzstyle{R} = [circle,draw=black,fill=red,minimum size=4pt,inner sep=0pt]

\begin{tikzpicture}[scale=0.14,auto=left]
  \draw (3,0) node[B] (0) [label=right:] {};
  \draw (0,3) node[B] (a) [label=left:{\scriptsize $a$}] {};
  \draw (3,3) node[B] (b) [label=left:{\scriptsize $b$}] {};
  \draw (6,3) node[B] (c) [label=right:{\scriptsize $c$}] {};
  \draw (3,6) node[B] (d) [label=left:{\scriptsize $d$}] {};
  \draw (3,8) node[B] (1) [label=left:] {};

\foreach \from/\to in {0/a,0/b,0/c,a/d,b/d,c/d,d/1}
    \draw[-] (\from) -- (\to);
\end{tikzpicture} &   

\tikzstyle{B} = [circle,draw=black,fill=black,minimum size=2pt,inner sep=0pt]
\tikzstyle{W} = [circle,draw=white,fill=white,minimum size=4pt,inner sep=0pt]
\tikzstyle{G} = [circle,draw=black,fill=green,minimum size=4pt,inner sep=0pt]
\tikzstyle{R} = [circle,draw=black,fill=red,minimum size=4pt,inner sep=0pt]

\begin{tikzpicture}[scale=0.14,auto=left]
  \draw (0,3) node[B] (a) [label=left:{\scriptsize $abc$}] {};
  \draw (-9,7) node[W] [] {};

\end{tikzpicture}

& $\alg M_3$ & $\mathbf{2}$ \\
 \hline

  $\alg B_2$ & \tikzstyle{B} = [circle,draw=black,fill=black,minimum size=2pt,inner sep=0pt]
\tikzstyle{W} = [circle,draw=white,fill=white,minimum size=4pt,inner sep=0pt]
\tikzstyle{G} = [circle,draw=black,fill=green,minimum size=4pt,inner sep=0pt]
\tikzstyle{R} = [circle,draw=black,fill=red,minimum size=4pt,inner sep=0pt]

\begin{tikzpicture}[scale=0.14,auto=left]

    \draw (13,8) node[B] (0d) [label=right:] {};
  \draw (10,5) node[B] (ad) [label=left:{\scriptsize $a$}] {};
  \draw (13,5) node[B] (bd) [label=left:{\scriptsize $b$}] {};
  \draw (16,5) node[B] (cd) [label=right:{\scriptsize $c$}] {};
  \draw (13,2) node[B] (dd) [label=left:{\scriptsize $d$}] {};
  \draw (13,0) node[B] (1d) [label=left:] {};

\foreach \from/\to in {0d/ad,0d/bd,0d/cd,ad/dd,bd/dd,cd/dd,dd/1d}
    \draw[-] (\from) -- (\to);
\end{tikzpicture} &   

\tikzstyle{B} = [circle,draw=black,fill=black,minimum size=2pt,inner sep=0pt]
\tikzstyle{W} = [circle,draw=white,fill=white,minimum size=4pt,inner sep=0pt]
\tikzstyle{G} = [circle,draw=black,fill=green,minimum size=4pt,inner sep=0pt]
\tikzstyle{R} = [circle,draw=black,fill=red,minimum size=4pt,inner sep=0pt]

\begin{tikzpicture}[scale=0.14,auto=left]
  \draw (0,3) node[B] (a) [label=left:{\scriptsize $abc$}] {};
    \draw (-9,7) node[W] [] {};

\end{tikzpicture}

& $\alg M_3$ & $\mathbf{2}$ \\
 \hline

 $\alg B_3$ & \tikzstyle{B} = [circle,draw=black,fill=black,minimum size=2pt,inner sep=0pt]
\tikzstyle{W} = [circle,draw=white,fill=white,minimum size=4pt,inner sep=0pt]
\tikzstyle{G} = [circle,draw=black,fill=green,minimum size=4pt,inner sep=0pt]
\tikzstyle{R} = [circle,draw=black,fill=red,minimum size=4pt,inner sep=0pt]

\begin{tikzpicture}[scale=0.14,auto=left]
  \draw (2,0) node[B] (0) [label=right:] {};
  \draw (0,2) node[B] (a) [label=left:{\scriptsize $a$}] {};
  \draw (4,3) node[B] (b) [label=right:{\scriptsize $b$}] {};
  \draw (0,4) node[B] (c) [label=left:{\scriptsize $c$}] {};
  \draw (2,6) node[B] (d) [label=left:{\scriptsize $d$}] {};
  \draw (2,8) node[B] (1) [label=left:] {};

\foreach \from/\to in {0/a,0/b,a/c,b/d,c/d,d/1}
    \draw[-] (\from) -- (\to);
\end{tikzpicture} &  

\tikzstyle{B} = [circle,draw=black,fill=black,minimum size=2pt,inner sep=0pt]
\tikzstyle{W} = [circle,draw=white,fill=white,minimum size=4pt,inner sep=0pt]
\tikzstyle{G} = [circle,draw=black,fill=green,minimum size=4pt,inner sep=0pt]
\tikzstyle{R} = [circle,draw=black,fill=red,minimum size=4pt,inner sep=0pt]

\begin{tikzpicture}[scale=0.14,auto=left]
  \draw (0,3) node[B] (a) [label=left:{\scriptsize $abc$}] {};
    \draw (-9,7) node[W] [] {};

\end{tikzpicture}

& $\alg 2$ & $\mathbf{2}$ \\
 \hline
 $\alg B_4$ & \tikzstyle{B} = [circle,draw=black,fill=black,minimum size=2pt,inner sep=0pt]
\tikzstyle{W} = [circle,draw=white,fill=white,minimum size=4pt,inner sep=0pt]
\tikzstyle{G} = [circle,draw=black,fill=green,minimum size=4pt,inner sep=0pt]
\tikzstyle{R} = [circle,draw=black,fill=red,minimum size=4pt,inner sep=0pt]

\begin{tikzpicture}[scale=0.14,auto=left]

  \draw (12,8) node[B] (0d) [label=right:] {};
  \draw (10,6) node[B] (ad) [label=left:{\scriptsize $a$}] {};
  \draw (14,5) node[B] (bd) [label=right:{\scriptsize $b$}] {};
  \draw (10,4) node[B] (cd) [label=left:{\scriptsize $c$}] {};
  \draw (12,2) node[B] (dd) [label=left:{\scriptsize $d$}] {};
  \draw (12,0) node[B] (1d) [label=left:] {};

\foreach \from/\to in {0d/ad,0d/bd,ad/cd,bd/dd,cd/dd,dd/1d}
    \draw[-] (\from) -- (\to);
\end{tikzpicture}&

\tikzstyle{B} = [circle,draw=black,fill=black,minimum size=2pt,inner sep=0pt]
\tikzstyle{W} = [circle,draw=white,fill=white,minimum size=4pt,inner sep=0pt]
\tikzstyle{G} = [circle,draw=black,fill=green,minimum size=4pt,inner sep=0pt]
\tikzstyle{R} = [circle,draw=black,fill=red,minimum size=4pt,inner sep=0pt]

\begin{tikzpicture}[scale=0.14,auto=left]
  \draw (0,3) node[B] (a) [label=left:{\scriptsize $abc$}] {};
    \draw (-9,7) node[W] [] {};

%\foreach \from/\to in {0/a,0/b,0/c,a/d,b/d,c/d}
 %   \draw[-] (\from) -- (\to);
\end{tikzpicture}

& $\alg 2$ & $\mathbf{2}$ \\
 \hline

\end{tabular}
  
    \caption{The $5$- and $6$-element nondistributive lattices and their medians. Note that $\alg A_1^{\alg 2}$ denotes not tha direct square of $\alg A_1$, but the lattice of monotone $\mathbf{2}\to\alg A_1$ maps. In particular, $|\OM\alg A_1|=64$.} 
    \label{6elementsTable}
\end{table}
\newpage

\bibliographystyle{unsrt} 
\bibliography{resources}

\end{document}